\documentclass[a4paper,11pt]{amsart}
\usepackage{amssymb,latexsym,amsmath,amsthm}
\usepackage{enumerate}
\usepackage{hyperref}

\DeclareMathOperator{\rk}{rk}
\DeclareMathOperator{\CRdim}{CR-dim}
\DeclareMathOperator{\CRcodim}{CR-codim}

\newcommand{\bC}{\mathbb{C}}
\newcommand{\bR}{\mathbb{R}}

\newcommand{\Id}{\mathrm{Id}}
\newcommand{\DGL}{\Delta_{GL}}
\newcommand{\DGLb}{\bar{\Delta}_{GL}}
\newcommand{\DR}{\Delta_{R}}

\DeclareMathOperator\vol{vol}

\newtheorem{Th}{Theorem}[section]
\newtheorem{Prop}[Th]{Proposition}
\newtheorem{Cor}[Th]{Corollary}
\newtheorem{Lem}[Th]{Lemma}
\theoremstyle{definition}
\newtheorem{Def}[Th]{Definition}
\newtheorem{Rem}[Th]{Remark}
\newtheorem{Not}[Th]{Notation}

\newcommand{\bt}{\begin{Th}\ \ }
\newcommand{\et}{\end{Th}}
\newcommand{\bp}{\begin{Prop}\ \ }
\newcommand{\ep}{\end{Prop}}
\newcommand{\bc}{\begin{Cor}\ \ }
\newcommand{\ec}{\end{Cor}}
\newcommand{\bl}{\begin{Lem}\ \ }
\newcommand{\el}{\end{Lem}}
\newcommand{\bd}{\begin{Def}\ \ }
\newcommand{\ed}{\end{Def}}
\newcommand{\br}{\begin{Rem}\ \ }
\newcommand{\er}{\end{Rem}}

\newcommand{\arr}{\begin{array}{rlll}}
\newcommand{\ea}{\end{array}}

\numberwithin{equation}{section}

\newcommand{\di}{\mathrm{d}}

\newcommand{\ii}{\mathrm{i}}
\newcommand{\ee}{\mathrm{e}}
\begin{document}
\title[Isometric pluriharmonic immersions of CR manifolds]{Isometric and CR pluriharmonic immersions\\ of three dimensional CR manifolds \\in Euclidean spaces}
\author[A.~Altomani]{Andrea Altomani}
\address{University of Luxembourg\\ Research Unit in Mathematics\\ rue Coudenhove-Kalergi 6\\ L-1359 Luxembourg}
\email{\texttt{andrea.altomani@uni.lu}}
\author[M.-A.~Lawn]{Marie-Am\'elie Lawn}\thanks{The second author has been supported by SNF grant FN 20-126689/1.}
\address{University of Neuchatel\\Institute of Mathematics\\Rue Emile Argand 11\\CH-2000 Neuchatel Switzerland}
\curraddr{Philipps Universit\"a{}t Marburg\\Fb.~12 -- Mathematik und Informatik\\ Hans-Meerwein-Stra\ss{}e\\ D-35032 Marburg}
\email{amelielawn@mathematik.uni-marburg.de}

\begin{abstract}
Using a bigraded differential complex depending on the CR and pseudohermitian structure, we give a characterization of three-dimensional strongly pseudoconvex pseudo-hermitian CR manifolds isometrically immersed in Euclidean space $\mathbb{R}^n$ in terms of an integral representation of Weierstra\ss{} type. Restricting to the case of immersions in $\mathbb{R}^4$, we study harmonicity conditions for such immersions and give a complete classification of CR-pluriharmonic immersions.
\end{abstract}


\date{}
\maketitle

\section*{Introduction}
The relationship between complex analysis and conformal geometry has been studied for a long time.
Historically, one of the first results is due to Weierstra\ss, who gave the well-known characterization of minimal surfaces in $\bR^3$ in the following way. Given a pair $(h,g)$ consisting of a holomorphic and a meromorphic function, the formula
\[
f(x,y)=\Re\int \big[(1 - g^2(z))h(z),\,(1+g^2(z))h(z),\,2g(z)h(z)\big]\,\di z, \tag{*}
\]
with $z=x+iy$ some complex coordinate, gives a local parametrization of a minimal surface in Euclidean three-space. Conversely every minimal surface can be parametrized in this way with respect to isothermal coordinates.
Equivalently, a minimal surface is an integral of the real part of a $\bC^3$-valued $1$-form that is holomorphic and isotropic in the bilinear inner product of $\bC^3$, and the forms of this type can be parametrized by the pair of functions $(h,g)$.

Recently, this approach was generalized to immersions of higher-dimensional complex manifolds
$M^{2m}$ into the Euclidean space $\mathbb{R}^n$, considering pluriminimal immersions, which means that the restriction of the immersion to any smooth
complex curve is minimal in $\mathbb{R}^n$ (see for example \cite{E}). These pluriminimal submanifolds have many analogies with the case of
minimal surfaces. In particular they admit associated families of isometric immersions and an integral representation of Weierstra\ss{} type as follows (see \cite{APS}).
Let $M^{2m}$ be a complex manifold, and $f\colon M^{2m}\to\bR^n$ an immersion which is pluriconformal (i.e. it is conformal when restricted to any holomorphic curve). The $\bC^n$-valued $(1,0)$-form $\omega=\partial f$ determines $f$ up to a constant, as we have that
\[
f(x)=\int_{x_0}^x\Re\omega+f(x_0), \tag{**}
\]
for all $x\in M$, and it satisfies the condition $\omega\cdot\omega=0$, where $"\cdot"$ is the bilinear symmetric product for $\bC^n$-valued forms.
 Conversely, every $\bC^n$-valued $(1,0)$-form $\omega$, such that $\Re\omega$ is exact and $\omega\cdot\omega=0$, defines via (**) a pluriconformal map $f\colon M\to \bR^n$. This map is pluriharmonic (and hence defines a pluriminimal immersion, if it is of maximal rank) if and only if $\omega$ is closed and holomorphic. Moreover, if it is an immersion then the pullback metric on $M$ is K\"ahler. In higher dimension however the parametrization of closed holomorphic isotropic $1$-forms in terms of a $k$-tuple of holomorphic or meromorphic functions is not available.

Our aim is to generalize this construction to immersions of odd-dimensional real manifolds with a strongly pseudoconvex CR structure of hypersurface type. In the present paper we consider the lowest dimensional case, namely three-dimensional CR manifolds immersed in $\bR^n$, with special attention to the case of immersions in $\bR^4$.

Strongly pseudoconvex CR manifolds admit a canonical family of metrics, the Webster metrics, parametrized by the choice of a pseudohermitian structure, i.e.\  a contact form compatible with the CR structure. All these metrics are conformally equivalent when restricted to the contact distribution. It is natural to replace the conformality condition of \cite{APS} with the condition that there exists a Webster metric for which the immersion is isometric.

There is an extensive literature about CR immersions of CR manifolds into complex spaces, with and without metric conditions (see e.g.~\cite{YK}).
In contrast, we would like to emphasize that the immersions we consider are \emph{real}, and not complex. In particular these immersions are not required to be CR maps. Later on we will restrict to immersions that are CR pluriharmonic, i.e. real parts of CR maps.


In \cite{Ru}, M.\ Rumin defines a differential complex for contact manifolds depending on the contact form $\theta$.
Using the complex structure on the subbundle
$T^{1,0}M\subset T^{\mathbb{C}}M$ defining the CR structure, it is then possible to introduce a bigrading of Rumin's complex following
Garfield and Lee \cite{GL}. This yields a double complex, which  generalizes the Dolbeault complex.
This construction, together with basics of CR geometry  will be recalled in section 1. The condition of pseudo-convexity allows us to choose $\theta$ in such a way that the Levi form is positive definite and can be associated to a Riemannian metric (the Webster metric). Hence we can define formal adjoint operators for the differentials of the Rumin complex and the Garfield-Lee complex, as well as corresponding Laplacians $\DR$ and $\DGL$.

In section 2 we compute explicitly the differentials of these complexes and their Laplacians for three dimensional strongly pseudoconvex pseudohermitian CR manifolds in all bidegrees with respect to a pseudohermitian local frame. This allows us to give the main result of section 3: we find an integral representation of Weierstra\ss{} type for isometric immersions of such manifolds in the Euclidean space $\mathbb{R}^n$, expressing the integrability, conformality and isometricity conditions in terms of properties of forms in the Garfield-Lee complex.

With the aim of considering an analog of pluriminimal immersions for odd dimensional manifolds, we consider in section 4 CR pluriharmonic immersions of three-dimensional CR manifolds into $\mathbb{R}^n$. Such immersions are not minimal, but satisfy a variant of a constant mean curvature condition.

For $n=4$, it turns out that the situation is more restrictive than in the classical case:
in section 5 we give a complete classification of isometric CR pluriharmonic immersions of three-dimensional strongly pseudoconvex pseudohermitian CR manifolds into $\bR^4$, showing that the standard CR embeddings of the sphere and the cylinder in $\bC^2$ are essentially the only two examples. 

For both these examples the shape operator commutes with the partial complex structure on $M$. 
This property is an analog of ``isotropic'' or ``circular''  immersions of K\"ahler manifolds. 

We would like to thank the anonymous referee for informing us about the existing literature on this topic (see e.g.~\cite[Ch.~16]{DO} and references therein), as well as for several helpful remarks.

\section{CR manifolds and pseudohermitian structures}
We recall the general definitions of CR manifolds, pseudohermitian structures, the Webster metric, and the differential complexes of Rumin and Garfield-Lee.

\subsection{CR manifolds and pseudohermitian structures}
\begin{Def}
A \emph{CR manifold} is the datum of a smooth manifold $M$, and of a complex subbundle $T^{1,0}M\subset T^\bC M$ of the complexification $T^\bC M=\bC\otimes TM$ of the tangent bundle of $M$, such that
\begin{gather*}
 T^{1,0}M\cap\overline{T^{1,0}M}=0,\\
[\Gamma(T^{1,0}M),\Gamma(T^{1,0}M)]\subset\Gamma(T^{1,0}M).
\end{gather*}
We denote $T^{0,1}M=\overline{T^{1,0}M}$.

The \emph{CR dimension} of $M$ is $\CRdim M=\rk_\bC T^{1,0}M$, and the \emph{CR codimension} of $M$ is $\CRcodim M=\dim_\bR M-2\CRdim M$.

A CR manifold of codimension one is said to be of \emph{hypersurface type}.
\end{Def}

Real manifold are in a natural way CR manifolds of CR dimension zero, and complex manifolds are CR manifolds of CR codimension zero.

\begin{Def}
A \emph{CR map} between two CR manifolds $(M,T^{1,0}M)$ and  $(N,T^{1,0}N)$ is a smooth map $f\colon M\to N$ such that $f_*(T^{1,0}M)\subset T^{1,0}N$.

A \emph{CR function} on a CR manifold $(M,T^{1,0}M)$ is a CR map from $M$ to $\bC$. Equivalently, a smooth function $f\colon M\to\bC$ is CR if and only if $\bar Zf=0$ for all $Z\in T^{1,0}M$.

A \emph{CR pluriharmonic function} on a CR manifold $(M,T^{1,0}M)$ is a smooth function $f\colon M\to\bR$ that is locally the real part of a CR function.

Vector valued CR and CR pluriharmonic functions are defined in the natural way.
\end{Def}


\begin{Def}
A \emph{pseudohermitian structure} on a CR manifold $(M,T^{1,0}M)$ of hypersurface type is a nowhere vanishing real $1$-form $\theta$ such that $\ker\theta=T^{1,0}M\oplus T^{0,1}M$.

A CR manifold with a pseudohermitian structure is a \emph{pseudohermitian CR manifold}.
\end{Def}
Notice that any two pseudohermitian structures $\theta$, $\theta'$ are related by $\theta'=\lambda\theta$ for some nonvanishing smooth function $\lambda$.

\begin{Def}
Let $(M,T^{1,0}M)$ be a CR manifold of hypersurface type, and $\theta$ a pseudohermitian structure on $M$. The \emph{Levi form} associated to $\theta$ is the hermitian symmetric form on $T^{1,0}M$:
\begin{align*}
\mathcal{L}_{\theta}\colon T^{1,0}M\times T^{1,0}M&\longrightarrow\bC\\
  (Z,W)&\longmapsto\mathcal{L}_{\theta}(Z,W)=\ii\di\theta(Z,\bar W).
\end{align*}
If $x$ is a point of $M$ and the Levi form at $x$ is nondegenerate (resp. definite), then $M$ is said to be \emph{Levi nondegenerate at $x$} (resp. \emph{strongly pseudoconvex at $x$}). A CR manifold is \emph{Levi nondegenerate} (resp. \emph{strongly pseudoconvex}) if it is Levi nondegenerate (resp. strongly pseudoconvex) at every point.
\end{Def}

\begin{Rem}
A pseudohermitian structure is a (nondegenerate) contact form on $M$ if and only if the Levi form is nondegenerate.
\end{Rem}

\begin{Rem}
Levi nondegeneracy and strong pseudoconvexity are independent from the choice of the pseudohermitian structure. Moreover, by replacing $\theta$ by $-\theta$ if needed, we can assume that a strongly pseudoconvex manifold has a positive definite Levi form.
\end{Rem}

\begin{Def}
Let $(M,T^{1,0}M)$ be a Levi nondegenerate CR manifold of hypersurface type, and $\theta$ a pseudohermitian structure on $M$. The unique vector field $T$ on $M$ such that
\[
\theta(T)=1,\qquad i_T(\di\theta)=0
\]
is the \emph{Reeb vector field}.
\end{Def}


\begin{Def}
Let $M$ be a Levi nondegenerate CR manifold of hypersurface type, and $\theta$ a pseudohermitian structure on $M$. The \emph{Webster pseudo-Riemannian metric} associated to $\theta$ is the symmetric nondegenerate bilinear form
\[ g_{\theta}\colon TM\times TM\longrightarrow\bR \]
defined by
\begin{align*}
g_\theta(X,Y)=\Re\mathcal{L}_\theta(Z,W),\quad g_\theta(X,T)=0,\quad g_\theta(T,T)=1
\end{align*}
where
\[X=\frac{Z+\bar Z}{\sqrt{2}},\quad Y=\frac{W+\bar W}{\sqrt{2}},\quad Z,W\in T^{1,0}M \]
and $T$ is the Reeb vector field.

If $M$ is strongly pseudoconvex and $\theta$ is chosen so that $\mathcal{L}_{\theta}$ is positive definite, then the Webster pseudo-Riemannian metric is positive definite and is called the \emph{Webster metric}.
\end{Def}
\begin{Not}\label{not:product}
We will denote by
\[ \langle\,\cdot\,.\,\cdot\,\rangle_\theta\colon T^\bC M\times T^\bC M\to\bC,\]
the bilinear symmetric extension of the Webster metric $g_\theta$ to  $T^\bC M$  and by
\[ (\,\cdot\,.\,{\cdot}\,)_\theta\colon T^\bC M\times T^\bC M\to\bC,\]
the hermitian symmetric one. We adopt the convention that $(\,\cdot\,.\,{\cdot}\,)_\theta$ is linear in the first variable and antilinear in the second one. Finally we set:
\[ \|\,\cdot\,\|^2_\theta=(\,\cdot\,.\,{\cdot}\,)_\theta. \]
We will omit the subscript $\theta$ when the choice of the pseudohermitian structure is clear from the context.
\end{Not}
\begin{Def}
If $\lambda>0$ is a constant and $\theta'=\lambda \theta$ then the corresponding Webster metrics $g_\theta$ and $g_{\theta'}$ are \emph{pseudo-homothetic}. More generally, a CR map $\phi\colon M\to M'$ between two pseudohermitian CR manifolds $(M,T^{1,0}M,\theta)$ and $(M',T ^{1,0}M',\theta')$ is a \emph{pseudo-homothety} if $\phi$ is a CR diffeomorphism and $\phi^*(\theta')=\lambda\theta$ for a constant $\lambda$.
\end{Def}

\subsection{The Rumin and Garfield-Lee complexes}
Following Rumin \cite{Ru} and Garfield and Lee \cite{GL}, we define a single and a double complex on $M$. We refer to those papers for the proof of the statements in this section.

 Let $(M,T^{1,0}M)$ be a $(2n+1)$-dimensional CR manifold of hypersurface type with nondegenerate Levi form and $\theta$ a pseudohermitian structure on $M$.
Denote by $\Omega(M)=\sum_k\Omega^k(M)$ the exterior algebra of complex differential forms on $M$, by $\mathcal{I}\subset
\Omega(M)$ the ideal generated by $\theta$ and $\di\theta$, and
define the annihilator of $\mathcal{I}$
\[ \mathcal{I}^{\perp}=\{\alpha\in\Omega(M)\mid \alpha\wedge\omega=0\ \forall\omega\in\mathcal{I}\} \]
and, for $k\in\mathbb{N}$, the complex vector bundles
\[
E^k=\Omega^k(M)/(\mathcal{I}\cap\Omega^k(M)),\qquad
F^k=\mathcal{I}^{\perp}\cap \Omega^k(M).
\]
Notice that
\[
 E^k=0,\,k > n, \qquad F^k=0,\, k\leq n.
 \]

The ideal $\mathcal{I}$ satisfies $\di\mathcal{I}\subset \mathcal{I}$.
It follows that the exterior derivative $\di\colon \Omega(M)\to \Omega(M)$ defines operators
\[
d\colon E^{k}\rightarrow E^{k+1},\qquad  d\colon F^{k}\rightarrow F^{k+1}.
\]

We define an operator $D\colon E^n\to F^{n+1}$ as follows.
For $\omega\in \Omega^n(M)$, there exists $\alpha\in\Omega^{n-1}(M)$ such that $\di(\omega+\theta\wedge\alpha)\in
F^{n+1}$. We set
\[ D[\omega]:=\di(\omega+\theta\wedge\alpha).\]
The operator $D$ is a second order differential operator.

\begin{Def}
The \emph{Rumin complex} of $(M,\theta)$ is the complex
\[
0\rightarrow E^0 \stackrel{d}{\rightarrow}
E^1 \stackrel{d}{\rightarrow}\ldots\stackrel{d}{\rightarrow}
E^n\stackrel{D}{\rightarrow}
F^{n+1}\stackrel{d}{\rightarrow}\ldots\stackrel{d}{\rightarrow}
F^{2n+1} \rightarrow 0.
\]
\end{Def}
\begin{Rem}
In \cite{Ru} it is proved that the sequence
\[
0\rightarrow\mathbb{R}\hookrightarrow E^0 \stackrel{d}{\rightarrow}
E^1 \stackrel{d}{\rightarrow}\ldots\stackrel{d}{\rightarrow}
E^n\stackrel{D}{\rightarrow}
F^{n+1}\stackrel{d}{\rightarrow}\ldots\stackrel{d}{\rightarrow}
F^{2n+1} \rightarrow 0,
\]
is locally exact.
\end{Rem}

\begin{Rem}
The Rumin complex only depends on the (possibly degenerate) contact form $\theta$, and not on the CR structure of $M$. Moreover, if $\theta$ is replaced, via a pseudo-homothety, by a constant multiple, the Rumin complex remains the same.
\end{Rem}

\begin{Not}
For the Rumin complex we will also use the notation  $R^k$ for the nonzero bundle $E^{k}$ or $F^k$.
\end{Not}

It is possible to introduce a bigrading on the Rumin complex in the following way:
consider the complexified Rumin bundles
\[ \mathbb{E}^k=E^k\otimes \mathbb{C},\quad \mathbb{F}^k=F^k\otimes \mathbb{C},\]
and define
\begin{align*}
{E}^{p,q}&:=\{[\omega]\in\mathbb{E}^{p+q}|\,\alpha|_{T^{1,0}M\oplus T^{0,1}M} \text{
is of type $(p,q)$ for some $\alpha\in[\omega]$} \}\\
{F}^{p,q}&:=\bigg\{\omega\in\mathbb{F}^{p+q}\,\Big|\,
\begin{gathered}\text{$i_X(\omega)|_{T^{1,0}M\oplus T^{0,1}M}$
 is of type $(p-1,q)$}\\
\text{for any $X\notin T^{1,0}M\oplus T^{0,1}M$}\end{gathered}
\bigg\}.
\end{align*}
where a form on $T^{1,0}M\oplus T^{0,1}M$ is said to be of type $(p,q)$ if it vanishes when applied to more than $p$ vectors in $T^{1,0}M$ or to more than $q$ vectors in $T^{0,1}M$.

For the Garfield-Lee complex too we will denote by $R^{p,q}$ the nonzero bundle ${E}^{p,q}$ or ${F}^{p,q}$.
Then we have
\begin{align*}
\mathbb{E}^k&=\bigoplus_{p+q=k}E^{p,q},\qquad
\mathbb{F}^k=\bigoplus_{p+q=k}F^{p,q},\\
d&\colon E^{p,q}\longrightarrow E^{p,q+1}\oplus E^{p+1,q},\\
d&\colon F^{p,q}\longrightarrow F^{p,q+1}\oplus F^{p+1,q},\\
D&\colon E^{p,q}\longrightarrow F^{p,q+1}\oplus F^{p+1,q}\oplus
F^{p+2,q-1}.
\end{align*}
and we define, for $0\leq p,q\leq2n+1$, by projection on
the subbundles, the operators
\begin{align*}
 d':=\pi^{p+1,q}\circ d&\colon R^{p,q}\longrightarrow R^{p+1,q},\\
 d'':=\pi^{p,q+1}\circ d&\colon R^{p,q}\longrightarrow R^{p,q+1},
 \end{align*}
 if $p+q\neq n$ and
\begin{align*}
D':=\pi^{p+1,q}\circ D&\colon E^{p,q}\longrightarrow F^{p+1,q},\\
D'':=\pi^{p,q+1}\circ D&\colon E^{p,q}\longrightarrow F^{p,q+1},\\
D^+:=\pi^{p+2,q-1}\circ D&\colon E^{p,q}\longrightarrow F^{p+2,q-1}.
\end{align*}
if $p+q=n$.

\begin{Def}
The \emph{Garfield-Lee complex} of a pseudohermitian CR manifold $(M, T^{1,0}M, \theta)$ is given by the spaces $E^{p,q}$ and $F^{p,q}$ and by the differential operators $d'$, $d''$, $D'$, $D''$, and $D^+$.

The Garfield-Lee complex is a complex in the following sense:
\begin{align*}
d'd'&=0, & d''d''&=0, & d'd''+d''d'&=0,\\
d''D''&=0, & D''d''&=0, & D'd''+D''d'&=0, & d'D''+d''D'&=0,\\
d'D^+&=0, & D^+d'&=0, & D'd'+D^+d''&=0, & d'D'+d''D^+&=0.
\end{align*}
\end{Def}

\begin{Rem}
The subcomplexes of the Garfield-Lee complex given by $d''$ and $D''$ are complexes in the usual sense, and they are locally exact at all positions $R^{p,q}$ with $q\geq1$.
\end{Rem}

\begin{Rem}
The Garfield-Lee complex is invariant under  a change of the pseudohermitian structure by pseudo-homothety.
\end{Rem}

\section{Three dimensional CR manifolds}
Let $(M, T^{1,0}M, \theta)$ be a three-dimensional Levi nondegenerate pseudohermitian CR manifold. Notice that in the three-dimensional case Levi nondegeneracy is equivalent to strong pseudoconvexity. We always assume that the pseudohermitian structure is chosen in such a way that the Levi form is positive definite.

The bundle $T^{1,0}M$ is one-dimensional, hence locally there exists a complex vector field $Z$ generating $T^{1,0}M$ at every point. Its complex conjugate
$\bar{Z}$ generates $T^{0,1}M$.

The Levi form $\mathcal{L}_\theta$ is completely determined by the value
\[\mathcal{L}_{\theta}(Z,Z)=\ii\di\theta(Z,\bar Z).\]
By the pseudoconvexity condition, $\mathcal{L}_{\theta}(Z,Z)$ is everywhere positive.
Upon replacing $Z$ with a scalar multiple we can assume that $\mathcal{L}_{\theta}(Z,{Z})=1$.

\begin{Def}
A \emph{pseudohermitian local frame} on a three-dimensional strongly pseudoconvex pseudohermitian CR manifold $(M, T^{1,0}M, \theta)$ is a local frame $(Z, \bar Z, T)$ with $\mathcal{L}_{\theta}(Z,{Z})=1$ and $T$ equal to the Reeb vector field.
\end{Def}

\begin{Rem}
A pseudohermitian local frame $(Z,\bar Z, T)$ is an orthonormal frame of $T^\bC M$ with respect to $(\,\cdot\,.\,{\cdot}\,)_\theta$, and $Z$, $\bar Z$ satisfy $\langle Z,Z\rangle_\theta=\langle \bar Z,\bar Z\rangle_\theta=0$.
The real vectors
\[ X=\frac{Z+\overline{Z}}{\sqrt 2},\qquad Y=\frac{\ii(Z-\overline{Z})}{\sqrt 2} \]
together with $T$ are an orthonormal frame for $g_\theta$.
\end{Rem}

\begin{Not}
If $(Z,\bar Z, T)$ is a pseudohermitian local frame of $M$, we denote by $(\zeta,\bar\zeta, \theta)$ the local frame of $T^*{}^\bC M$ dual to $(Z,\bar Z, T)$ and we define three complex valued smooth functions $a$, $b$, and $c$ on $M$ by:
\begin{equation}
\begin{aligned}
\ii[Z,\bar Z]&=T+aZ+\bar a \bar Z, &\qquad a&=\ii\zeta[Z,\bar Z], \\
[Z,T]&= bZ+\bar c \bar Z, & b&=\zeta[Z,T],  \\
[\bar Z,T]&=cZ+\bar b\bar Z, & c&=\zeta[\bar Z,T].
\end{aligned}
\end{equation}\label{eq_for_a}\label{eq_for_b_barc}\label{eq_for_c_barb}
We also denote by $\vol$ the volume form $\zeta\wedge\bar\zeta\wedge\theta$.
\end{Not}

Straightforward computations yield:
\begin{equation}
\begin{aligned}
\di\theta&=\ii\,\zeta\wedge\bar\zeta, \\
\di\zeta&=\ii a \,\zeta\wedge\bar\zeta-b\,\zeta\wedge\theta-c\,\bar\zeta\wedge\theta, \\
\di\bar\zeta&=\ii\bar a\,\zeta\wedge\bar\zeta-\bar c\,\zeta\wedge\theta-\bar b\,\bar\zeta\wedge\theta.
\end{aligned}
\end{equation}
Moreover, from the Jacobi identity for $Z$, $\bar{Z}$ and $T$ we have
\begin{align}\label{eq_jacobi}
b+\bar b&=0,&
\ii Zc-\ii\bar Zb+Ta-ab-\bar ac&=0.
\end{align}

\begin{Rem}\label{rem:change_of_frame}
If $(Z,\bar Z, T)$ is a pseudohermitian local frame of $M$, then any other pseudohermitian local frame $(Z',\bar Z', T')$ can be locally obtained from $(Z,\bar Z, T)$ as:
\begin{align*}
{Z'}&=\ee^{-\ii v}Z,  & {\bar Z'}&=\ee^{\ii v}\bar Z, & {T'}&=T, \\
 {\bar\zeta'}&=\ee^{-\ii v}\bar\zeta, & {\zeta'}&=\ee^{\ii v}\zeta, & {\theta'}&=\theta.
\end{align*}
for some real valued function $v$ on $M$. The functions $a$, $b$, and $c$ in the new frame are then
\begin{align*}
 a'&=\ee^{\ii v}(a- \bar Zv), &
 b'&=b+\ii Tv, &
 c'&=\ee^{2\ii v}c.
\end{align*}
\end{Rem}
\begin{Def}
A change of frame as described in  Remark~\ref{rem:change_of_frame} is called a \emph{pseudohermitian change of frame}.
\end{Def}

\begin{Rem}
On any Levi nondegenerate pseudohermitian CR manifolds, there is a naturally defined linear connection $\nabla^\theta$, called the Tanaka-Webster connection. In this paper we will not use it, however we notice that the functions $a$, $b$, and $c$ are related to $\nabla^\theta$ and to its torsion $T^\theta$ by the following equations:
\begin{gather*} \nabla^{\theta}_ZZ=-\ii\bar aZ,\quad\nabla^{\theta}_{\bar{Z}}Z=iaZ,\quad \nabla^{\theta}_TZ=-bZ,\quad T^{\theta}(T,\bar{Z})=cZ.
\end{gather*}
\end{Rem}

\begin{Rem}
A strongly pseudoconvex pseudohermitian CR manifold is in a natural way a contact metric manifold. In the three dimensional case the converse is also true, since every almost CR structure is integrable. We recall that a contact metric manifold is Sasakian if the contact metric structure is normal (see e.g.~\cite{DT}). This is equivalent to the vanishing of the pseudohermitian torsion $\tau=T^\theta(T,\cdot)$. With our notation, $M$ is a Sasakian manifold if and only if $c$ vanishes identically.
\end{Rem}

\subsection{The Rumin and Garfield-Lee complexes}

We describe now in detail the Rumin and Garfield-Lee complexes in the three-dimensional case. Let $(M, T^{1,0}M, \theta)$ be a three-dimensional strongly pseudoconvex pseudohermitian CR manifold.
We have $\di\theta=\ii\zeta\wedge\bar{\zeta}$, and
\begin{align*}
\mathcal{I}&=\langle\theta, \zeta\wedge\bar{\zeta}, \theta\wedge\zeta,
\theta\wedge\bar{\zeta}, \vol \rangle, &
\mathcal{I}^{\perp}&=\langle \zeta\wedge\theta, \bar{\zeta}\wedge
\theta,\vol\rangle,
\end{align*}
and consequently the spaces occurring in the Rumin complex are
\begin{align*}
E^0&=\Omega^0(M)=\bC, &
E^1&=\Omega^1(M)/\langle\theta\rangle\simeq\langle \zeta, \bar{\zeta}\rangle, \\
F^2&=\langle\zeta\wedge\theta,\bar{\zeta}\wedge\theta \rangle, & F^3&=\Omega^3(M)=\langle\vol\rangle,
\end{align*}
and in the Garfield-Lee complex
\begin{gather*}
E^{0,0}=E^0=\bC,\\
\begin{alignedat}{2}
E^{1,0}&=\langle\zeta,\theta\rangle/\langle\theta\rangle\simeq\langle\zeta\rangle,&\qquad
E^{0,1}&=\langle\bar\zeta,\theta\rangle/\langle\theta\rangle\simeq\langle\bar\zeta\rangle,\\
F^{2,0}&=\langle\zeta\wedge\theta\rangle,&
F^{1,1}&=\langle\bar\zeta\wedge\theta\rangle,
\end{alignedat}\\
F^{2,1}=F^3=\langle\vol\rangle,
\end{gather*}
all other terms being $0$.

\begin{Prop}\label{prop:CRph_Rumin}
Let $(M,T^{1,0}M,\theta)$ be a three-dimensional Levi nondegenerate pseudohermitian CR manifold and $f\colon M\to\bR$ a smooth function. Then:
\[
 \text{$f$ is CR pluriharmonic}\Longleftrightarrow Dd'f=0\Longleftrightarrow D'd'f=0\Longleftrightarrow D''d'f=0. \]
\end{Prop}
\begin{proof}
 First we show that the last three conditions are equivalent, for a smooth real function $f$ on $M$. Decomposing $Dd'f$ into its $(1,1)$ and $(2,0)$ components we see that $Dd'f=0$ if and only if both $D'd'f$ and $D''d'f$ are zero. Moreover $Ddf=0$ and hence, taking the $(1,1)$ component we get $0=(Ddf)^{1,1}=D''d'f+D'd''f$. Since $f$ is real, we also have $d''f=\overline{d'f}$ and $D'd''f=\overline{D'd'f}$. It follows that $D''d'f=-\overline{D'd'f}$ and that $D'd'f=0$ if and only if $D''d'f=0$.

Let now $f$ be CR pluriharmonic. Since CR pluriharmonicity is a local property, we can assume that there is a CR function $g\colon M\to\bC$ with $f=g+\bar g$. Then $d'f=d'g+d'\bar g=d'g=dg$ and $Dd'f=Ddg=0$.

Finally assume that $f$ is a real function with $Dd'f=0$. By exactness of the (complexification of) Rumin complex, locally there exists a complex smooth function $g$ with $dg=d'f$, i.e. $d'g=d'f$ and $d''g=0$. The latter condition implies that $g$ is a CR function, and we have $d(g+\bar g)=d'g+d''\bar g=d'f+d''f=df$, hence $(g+\bar g)-f$ is a real constant, and $f$ is CR pluriharmonic.
\end{proof}

We explicitly compute the differentials in all degrees $(p,q)$ of the Garfield-Lee complex, with respect to a pseudohermitian local frame.

\paragraph{$(p,q)=(0,0)$} Let $f\in E^{0,0}$. Then $\di f=Zf\,\zeta+\bar{Z}f\,\bar{\zeta}+Tf\,\theta$ and
\begin{align}
d'f&=[Zf\,\zeta],&d''f&=[\bar{Z}f\,\bar{\zeta}].
\end{align}

\paragraph{$(p,q)=(1,0)$} Let $[\alpha]=[g\zeta]\in E^{1,0}$. Then
$ D([\alpha])=D([g\zeta])=\di(g\zeta+h\theta) $
for a function $h$ such that $\di(g\zeta+h\theta)\in F^2$. From
\begin{align*}
\di(g\zeta+h\theta)&=\bar{Z}g\,\bar{\zeta}\wedge\zeta+Tg\,\theta\wedge\zeta+g\,\di\zeta+Zh\,\zeta\wedge\theta+
\bar{Z}h\,\bar{\zeta}\wedge\theta+h\,d\theta\\
&=(-\bar{Z}g+\ii ag+\ii h)\,\zeta\wedge\bar{\zeta}+(-Tg-bg+Zh)\,\zeta\wedge\theta\\
  &\quad+(-cg+\bar{Z}h)\,\bar{\zeta}\wedge\theta.
\end{align*}
we get $(-\bar{Z}g+\ii ag+\ii h)=0$ and consequently
\begin{align*}
Zh&=-\ii Z\bar{Z}g-Z(ag), &\bar{Z}h&=-\ii\bar{Z}^2g-\bar{Z}(ag),
\end{align*}
finally giving
\begin{equation}\label{computation_D_for_(p,q)=(1,0)}
\begin{split}
 D'([g\zeta])&=-(Tg+bg+\ii Z(\bar{Z}g-\ii ag))\,\zeta\wedge\theta, \\
D''([g\zeta]) &= -(cg+\ii\bar{Z}(\bar{Z}g-\ii ag))\,\bar{\zeta}\wedge\theta.
\end{split}
\end{equation}

\paragraph{$(p,q)=(0,1)$} A similar calculation, for a form $[\alpha]=[\bar g\bar\zeta]\in E^{0,1}$ leads to
\begin{equation}
\begin{aligned}
D'([\bar g\bar\zeta])&=-(T\bar g+\bar b\bar g-\ii\bar Z(Z\bar g+\ii\bar a\bar g))\,\bar\zeta\wedge\theta \\
D^+([\bar g\bar\zeta])&=- (\bar c\bar g-\ii Z(Z\bar g+\ii\bar a\bar g))\,\zeta\wedge\theta.
\end{aligned}
\end{equation}
\paragraph{$(p,q)=(2,0)$} Let $\alpha=g\,\zeta\wedge\theta\in F^{2,0}$. We have then
 \begin{equation*}
 d''(g\,\zeta\wedge\theta)=\bar{Z}g\,\bar{\zeta}\wedge\zeta\wedge\theta+g\,\di\zeta\wedge\theta-g\,\zeta\wedge \di\theta
\end{equation*}
and consequently
\begin{equation}\label{eq:20}
 d''(g\,\zeta\wedge\theta)=-(\bar{Z}g-\ii ag)\,\vol.
 \end{equation}

\paragraph{$(p,q)=(1,1)$} Similarly, for $\alpha=\bar{g}\,\bar{\zeta}\wedge\theta\in F^{1,1}$ we obtain
\begin{equation}
d'(\bar{g}\,\bar{\zeta}\wedge\theta)=(Z\bar g+\ii\bar a\bar g)\,\vol.
\end{equation}

\begin{Cor}\label{cor:CRph_frame}
Let $(M,T^{1,0}M,\theta)$ be a three-dimensional Levi nondegenerate pseudohermitian CR manifold, $(Z,\bar Z, T)$ a pseudohermitian frame, and $f\colon M\to\bR$ a smooth function. The following are equivalent:
\begin{enumerate}[\rm (1)]
\item $f$ is CR pluriharmonic,
\item $TZf+bZf+\ii Z(\bar ZZf-\ii aZf)=0$,
\item $cZf+\ii\bar Z(\bar ZZf-\ii aZf)=0$. \qed
\end{enumerate}
\end{Cor}

\subsection{Laplacians}
The Webster metric on a strongly pseudoconvex CR manifold $M$ allows to construct an operator on the Garfield-Lee and Rumin complexes analogous to the Hodge $*$-operator.  More details can be found in \cite{GL}.
If $M$ is a strongly pseudoconvex pseudohermitian CR manifold of dimension 3, the $*$ operator is explicitly given, for a form $\alpha\in E^{p,q}$, ($0\leq p+q\leq 1$) by:
\begin{enumerate}[(i)]
\item if $\alpha=f\in E^{0,0}$, then $*\alpha=f \vol\in F^{2,1}$,
\item if $\alpha=[g \zeta]\in E^{1,0}$, then $*\alpha=g\,\zeta\wedge \theta\in F^{2,0}$,
\item if $\alpha=[\bar g \bar\zeta]\in E^{0,1}$, then $*\alpha=\bar g\,\bar\zeta\wedge \theta\in F^{1,1}$,
\end{enumerate}
and for a form $\alpha\in F^{p,q}$, ($2\leq p+q\leq 3$) by:
\begin{enumerate}[(i)]
\item if $\alpha=f \vol\in F^{2,1}$, then $*\alpha=f\in E^{0,0}$,
\item if $\alpha=g\,\zeta\wedge \theta\in F^{2,0}$, then $*\alpha=[g \zeta]\in E^{1,0}$,
\item if $\alpha=\bar g\,\bar\zeta\wedge \theta\in F^{1,1}$, then $*\alpha=[\bar g \bar\zeta]\in E^{0,1}$.
\end{enumerate}
Then $*$ is a linear isomorphism and $*^2=\Id$.

Let $\bar*$ be $*$ followed by complex conjugation.
The formal adjoint operators of the differentials $d,d',d''$ (and of $D,D',D''$ in the middle degrees) are:
\begin{align*}
\delta&:=d^*=(-1)^{p+q}*d*,\\
 \delta'&:=d'^*=(-1)^{p+q}*d''*=(-1)^{p+q}\bar*d'\bar*,
 \\ \delta''&:=d''^*=(-1)^{p+q}*d'*=(-1)^{p+q}\bar*d''\bar*.
\end{align*}

\begin{Def} The \emph{Garfield-Lee Laplacian} on the Garfield-Lee complex of a Levi nondegenerate pseudohermitian CR manifold $M$ of dimension $2n+1$ is defined, on each space $R^{p,q}$, by:
\[ \DGL = \begin{cases}
d''\delta''+\delta''d'' &\text{if $p+q\neq n, n+1$,}\\
(d''\delta'')^2+D''{}^*D''&\text{if $p+q=n$,}\\
D''D''{}^*+(\delta''d'')^2&\text{if $p+q=n+1$.}\\
\end{cases} \]
\end{Def}

\begin{Def} The \emph{Rumin Laplacian} on the Rumin complex of a Levi nondegenerate pseudohermitian CR manifold $M$ of dimension $2n+1$ is defined, on each space $R^{k}$, by:
\[ \DR:= \begin{cases}
d\delta+\delta d &\text{if $k\neq n, n+1$,}\\
(d\delta)^2+D{}^*D&\text{if $k=n$,}\\
DD{}^*+(\delta d)^2&\text{if $k=n+1$.}\\
\end{cases} \]
\end{Def}

\begin{Rem}
For functions on a three-dimensional CR manifold we have
\begin{align*}
 \DGL f&=\delta'' d'' f + d'' \delta'' f=\delta'' d'' f = -\ast d' \ast d'' f
  =-\ast d' \ast (\bar Zf\,\bar\zeta)\\&=-\ast d' (\bar Z f\,\bar\zeta\wedge\theta)
  =-\ast \big((Z\bar Z f+\ii\bar a\bar Zf)\,\vol\big)
\end{align*}
giving the explicit description:
\begin{align}
\begin{split}
\DGL f&=-(Z\bar Z +\ii \bar a\bar Z)f,\end{split}\label{eq:DGL_frame}\\
\DR f&=-Z\bar Z f-\ii \bar a\bar Zf-\bar Z Z  f +\ii a Z f.\label{eq:DR_frame}
\end{align}
From \eqref{eq_for_a} and \eqref{eq:DR_frame} we obtain:
\begin{align}
\DR f &=\DGL f+{\DGLb} f=2 \DGL f -\ii Tf, \label{relation_between_D_R_and_D_GL}\\
\DGL f&=\frac12(\DR f+\ii Tf).
\end{align}
\end{Rem}

A straightforward computation shows:
\begin{Lem}
If the pseudohermitian structure $\theta$ is changed via a pseudo-homothety to a constant multiple $\theta'=\lambda \theta$, then the Rumin and Garfield-Lee Laplacians change as:
\begin{align*}
\Delta'_R&=\lambda^{-2}\DR,&\Delta'_{GL}=\lambda^{-2}\DGL.\qed
\end{align*}
\end{Lem}

\section{Integral Weierstra\ss{} representation, integrability and isometricity conditions}\label{s:w}
 The object of our study are isometric immersions of a three-dimensional CR manifold into an $n$-dimensional  real Euclidean space.
 We will consider more extensively immersions in the $4$-dimensional space.

\begin{Not}
Consistently with Notation~\ref{not:product},
we will denote by
\[ \langle\,\cdot\,,\,\cdot\,\rangle\colon \bC^n \times \bC^n\to\bC,\]
the standard bilinear symmetric inner product of $\bC^n$  and by
\[ (\,\cdot\,,\,{\cdot}\,)\colon T^\bC M\times T^\bC M\to\bC,\]
the hermitian symmetric one, with the convention that $(\,\cdot\,,\,{\cdot}\,)$ is linear in the first variable and antilinear in the second one. We also set:
\[ \|\,\cdot\,\|^2_\theta=(\,\cdot\,,\,{\cdot}\,)_\theta. \]
 \end{Not}

In analogy with \cite[Theorem~1.1]{APS} we give an integral representation of Weierstra\ss{} type, similar to (**), for isometric and for CR-pluriharmonic immersions.
We characterize isometric immersions $f\colon M\to\bR^n$ of a three-dimensional strongly pseudoconvex pseudohermitan CR manifold $(M, T^{1,0}M,\theta)$ in terms of the forms $\omega=d' f\in E^{1,0}\otimes\bR^n$.

First we give conditions for the integrability of a form $\omega\in E^{1,0}$. We have:

\begin{Prop}\label{integrability_immersion}
Let $\omega$ be a form in $ E^{1,0}$. Locally there exists a real valued $f\in E^{0,0}$ such that $d'f=\omega$ if and only if
the following equivalent conditions hold true:
\begin{equation}\label{integrability_cond2}
D(\omega+\bar\omega)=0\Longleftrightarrow D'\omega=-D^+\bar{\omega}
\Longleftrightarrow  D''\omega=-D'\bar{\omega}
\end{equation}
\end{Prop}
 \begin{proof}
 From $Ddf=0$, separating the terms according to their bidegrees, we get \eqref{integrability_cond2}. The converse follows from the local exactness of the Rumin complex.
 \end{proof}
\begin{Rem}
By \eqref{computation_D_for_(p,q)=(1,0)} and \eqref{eq:20}, if there exists a nonzero form $\omega\in E^{1,0}$ that is $D$-closed and $\delta$-closed then the pseudohermitian structure on $M$ is Sasakian.
\end{Rem}

We can now give the characterizations of isometric immersions and of CR pluriharmonic isometric immersions

\begin{Th}\label{th:weierstrass}
Let $(M, T^{1,0}M,\theta)$ be a three-dimensional strongly pseudoconvex pseudohermitan CR manifold. There is a bijective correspondence, given by $\omega=d'f$, between local isometric immersions $f\colon M\to\bR^n$, up to translation, and forms $\omega\in E^{1,0}\otimes_\bC \bC^n$ satisfying the conditions:
\begin{enumerate}[\rm(1)]
\item $D''\omega=-D'\bar\omega$,
\item $\langle\omega,\omega\rangle=0$,
\item
$\|\omega\|^2=1$,
\item
$\|\delta(\omega-\bar\omega)\|^2=1$,
\item
$\delta\bar\omega\cdot\omega=0$.
\end{enumerate}
\end{Th}
In condition (5) we use the following notation: if $h=\sum h_i \otimes e_i\in E^{0,0}\otimes\bC^n$ and $\alpha=\sum \alpha_i\otimes e_i\in E^{p,q}\otimes\bC^n$, then $h\cdot\alpha=\sum h_i\alpha_i\in E^{p,q}$.

\begin{Rem}
If $(Z,\bar Z,T)$ is a pseudohermitian frame for $M$, and $(\zeta,\bar\zeta, \theta)$ the dual frame, setting $\omega=[\zeta]\otimes\phi$ for a function $\phi\colon M\to\bC^n$,  straightforward computations show that conditions (1)--(5) for $\omega$ are equivalent to the following conditions (1')--(5') for $\phi$.
\begin{enumerate}[\rm(1')]
\item $-c\phi+\ii\bar Z(\bar Z\phi-\ii a\phi)+T\bar\phi-b\bar\phi-\ii\bar Z(Z\bar\phi+\ii\bar a\bar\phi)=0$,
\item $\langle\phi,\phi\rangle=0$,
\item
$\|\phi\|^2=1$,
\item
$\|\bar Z\phi-\ii a\phi\|^2=1/2$,
\item
$\langle Z\bar\phi+\ii\bar a\bar\phi,\phi\rangle=0$.
\end{enumerate}
\end{Rem}

\begin{proof}[Proof of Theorem \ref{th:weierstrass}]
Condition (1) and Proposition~\ref{integrability_immersion} ensure the existence of a map $f\colon M\to\bR^n$ satisfying $d'f=\omega$.

Fix a pseudohermitian frame $(Z,\bar Z, T)$ for $M$ and let $\omega=\phi\zeta$. Clearly $f$ is isometric if and only if all the following conditions hold true:
\begin{equation*}
\begin{cases}
 (Zf,\bar Zf)=0,\\
(Tf,Zf)=0,\\
\|Tf\|^2=\|Zf\|^2=1.
 \end{cases}
\end{equation*}
Conditions (2) and (3) are easily seen to be equivalent to
\[ (Zf,\bar Zf)=0 , \qquad \|Zf\|^2=\|\bar Zf\|^2=1. \]

Recall that, if $\omega=d'f$, then $Tf=\ii\delta'd'f-\ii\delta''d''f=\ii\delta(\omega-\bar\omega)$. Then we have
\[ \|Tf\|^2=1,\qquad (Zf,Tf)=0, \]
if and only if $\|\delta(\omega-\bar\omega)\|^2=1$ and $\delta(\omega-\bar\omega)\cdot\omega=0$.

To complete the proof, we observe that $\delta\omega\cdot\omega=-\langle\bar Z\phi-\ii a\phi,\phi\rangle=0$, because $\langle\omega,\omega\rangle=\langle\phi,\phi\rangle=0$.
\end{proof}

\begin{Cor}\label{th:weierstrass2}
With the same notation,
there is a bijective correspondence between CR pluriharmonic local isometric immersions $f\colon M\to\bR^n$, up to translation, and forms $\omega\in E^{1,0}\otimes \bC^n$ satisfying the conditions {\rm(2), (3), (5)} of Theorem~\ref{th:weierstrass} and:
\begin{enumerate}
\item[\rm(1)] $D\omega=0$,
\item[\rm(4)]
$\|\delta \omega\|^2=1$.
\end{enumerate}
\end{Cor}
\begin{proof}
Condition (1) is equivalent to CR pluriharmonicity by Proposition~\ref{prop:CRph_Rumin}.
With the notation as in the previous Remark, condition (1) can be written as $\bar Z(\bar Z\phi-\ii a\phi)=-\ii c\phi$. Together with (2) and (5) this yields $\langle\delta\omega,\delta\omega\rangle=0$ and then condition (4) becomes equivalent to condition (4) in Theorem~\ref{th:weierstrass}.
\end{proof}

\section{Harmonicity conditions}
We want to characterize isometries from a three-dimensional strongly pseudoconvex pseudohermitian CR manifold into $\bR^n$ that satisfy an additional harmonicity condition. Reasoning as in the proof of Theorem~\ref{th:weierstrass}, one can see that the condition $\Delta_R f=0$ (known as pseudo-harmonicity in the literature) is incompatible with an isometric immersion. It is natural then to consider CR pluriharmonicity instead, which is analogous to the condition that $\partial f$ is a closed holomorphic form in \cite{APS}.

We recall that a map $f\colon M\to \bR^n$ is CR pluriharmonic if it is locally the real part of a CR map into $\bC^n$, or equivalently if all the components are CR pluriharmonic functions.
From Proposition~\ref{prop:CRph_Rumin}, Corollary~\ref{cor:CRph_frame}, and equation \eqref{eq:DGL_frame} we obtain:
\begin{Prop}
For a  map $f\colon M\to \bR^n$ the following conditions are equivalent:
\begin{enumerate}[\rm(1)]
\item
$f$ is CR pluriharmonic,
\item
$Dd'f=0$,
\item
$\bar Z(\DGL f)=\ii T\bar Zf-\ii b\bar Z f$,
\item $Z(\DGL f)=\ii\bar c\bar Zf$.\qed
\end{enumerate}
\end{Prop}

Let now $f\colon M\to\bR^n$ be an isometric immersion. We consider the following conditions on $f$:
\begin{enumerate}[(1)]
\item
$f$ is CR pluriharmonic;
\item
$\langle \DGL f,\Delta_{GL} f\rangle=0$;
\item
$\DR f$ is orthogonal to $TM$ and has constant length.
\end{enumerate}
In the special case $n=4$ we also consider the condition:
\begin{enumerate}[(1)]\setcounter{enumi}3
\item
$\DR f$ is a parallel section of the normal bundle of $M$.
\end{enumerate}

\begin{Prop}\label{prop:forward}
Let $M$ be a three-dimensional strongly pseudoconvex pseudohermitian CR manifold,  and  $f\colon M\to\bR^n$ an isometric immersion. Then we have
\[ (1)\Longrightarrow(2)\Longrightarrow(3). \]
\end{Prop}
\begin{proof}
First we observe that for any isometric immersion, condition (5) in Theorem~\ref{th:weierstrass} implies that $\DGL f$ is orthogonal to $Zf$ and $\bar Zf$:
\begin{align*}
\langle\Delta_{GL} f, Zf\rangle&=0,&\langle\Delta_{GL} f, \bar{Z}f\rangle&=0
\end{align*}

\paragraph{$(1)\Longrightarrow(2)$}
We note that $\DGL f=\delta\omega$, in the notation of section~\ref{s:w}. The implication then follows directly from Corollary~\ref{th:weierstrass2}.

\paragraph{$(2)\Longrightarrow(3)$}

Recall that $\DR f=2 \DGL f -\ii Tf$. From $(2)$ we get
\[
0=4\langle\DGL f,\Delta_{GL} f\rangle=(\|\DR f\|^2-\|Tf\|^2)+2\ii\langle\DR f,Tf\rangle
\]
thus it follows that $\DR f$ is orthogonal to $Tf$ and that $\|\DR f\|^2=\|Tf\|^2$, and hence $\|\DR f\|$ is constant  for an isometric immersion.\end{proof}

\begin{Lem}Under the hypotheses of Proposition~\ref{prop:forward}, if $n=4$ then
\[
(3)\Longleftrightarrow (4).
\]
\end{Lem}
\begin{proof}
If $\langle\Delta_{GL}f,\Delta_{GL}f\rangle=0$, then $\Delta_Rf$ is a section of the normal bundle of $M$, and it has constant length, i.e. it is parallel.
\end{proof}

\begin{Rem}
We observe that the Laplacian $\DR f$ is, for an isometric immersion $f$, essentially the mean curvature restricted to the contact distribution $T^{1,0}M+T^{0,1}M$. CR pluriharmonic isometric immersions can be then seen as special variants of constant mean curvature immersions.
\end{Rem}

We look now at conditions such that the reverse implications $(3)\Rightarrow(2)$ and $(2)\Rightarrow(1)$ hold true.
\begin{Prop}
If $f\colon M\to\mathbb R^n$ satisfies condition $(3)$, then it satisfies condition $(2)$ if and only if $\|\DR f\|=1$
\end{Prop}
\begin{proof}
The statement easily follows from the equation $4\langle\DGL f,\DGL f\rangle=(\|\DR f\|^2-\|Tf\|^2)+2\ii(\DR f,Tf)$  and the fact that for an isometric immersion $\| Tf\|=1$.
\end{proof}

\begin{Prop}
Assume that $(2)$ holds and $n=4$. Then $(1)$ holds if and only if the image of $\DGL$ is a totally isotropic submanifold of $\mathbb{C}^4$ with respect to the standard complex bilinear form, i.e.:
\begin{align*}
\langle Z\DGL f,Z\DGL f\rangle&=0,\\
\langle Z\DGL f,\bar Z\DGL f\rangle&=0,\\
\langle Z\DGL f,T\DGL f\rangle&=0,\\
\langle \bar Z\DGL f,\bar{Z}\DGL f\rangle&=0,\\
\langle\bar Z\DGL f,T\DGL f\rangle&=0,\\
\langle T\DGL f,T\DGL f\rangle&=0.
\end{align*}
\end{Prop}
\begin{proof}
We use the fact that , if $(2)$ holds, then $\{Z f,\bar Z f,2^{1/2}\DGL f,2^{1/2}\DGLb f\}$ is an orthonormal basis of $\bC^4$, and hence, for all vectors $v,w\in\bC^4$, we have:
\begin{align*}
(v,w)&=(v,Z f)(\overline{w, Z f}) + (v,\bar Z f)(\overline{w, \bar Z f})\\
&\quad + 2(v,\DGL f)(\overline{w, \DGL f}) + 2(v, \DGLb f)(\overline{w, \DGLb f}).
\end{align*}
Thus we obtain:
\begin{align*}
\langle Z\DGL f,Z\DGL \rangle&=0,\\
\langle Z\DGL f,\bar{Z}\DGL f\rangle&=\langle Z\DGL f,Z f\rangle\langle \bar{Z}\DGL f,\bar {Z} f\rangle,\\
\langle Z\DGL f,T\DGL f\rangle&=\langle Z\DGL f,Z f\rangle\langle T\DGL f,\bar Z f\rangle,\\
\langle\bar{Z}\DGL f,\bar{Z}\DGL f\rangle&=4\langle \bar{Z}\DGL f,\bar Z f\rangle,\\
\langle\bar{Z}\DGL f,T\DGL f\rangle&=2\langle T\DGL f,\bar Z f\rangle+\langle\bar{Z}\DGL f,\bar{Z} f\rangle\langle T\DGL f, Z f\rangle,\\
\langle T\DGL f,T\DGL f\rangle&=2\langle T\DGL f,\bar Z f\rangle\langle T\DGL f, Z f\rangle.
\end{align*}

If $(1)$ holds, then $\langle\bar Z\DGL f,\bar{Z} f\rangle=0$ and $\langle T\DGL f,\bar{Z} f\rangle=0$, obtaining immediately the conclusion.

Conversely, if $\langle\bar Z\DGL f,\bar{Z}\DGL f\rangle=0$ and $\langle\bar Z\DGL f,T\DGL f\rangle=0$ then  $\langle\bar Z\DGL f,\bar{Z} f\rangle=0$ and $\langle T\DGL f,\bar{Z} f\rangle=0$. Then we have:
\begin{align*}
\langle Z\DGL f-\ii\bar c\bar Zf,\DGLb f\rangle&=0,\\
\langle Z\DGL f-\ii\bar c\bar Zf,\DGL f\rangle&=0,\\
\langle Z\DGL f-\ii\bar c\bar Zf,\bar{Z}f\rangle&=\langle Z\DGL f,\bar{Z}f\rangle=-\langle\DGL f,Z\bar{Z}f\rangle\\
  &=\langle\DGL f,\Delta_{GL} f\rangle=0,\\
\langle Z\DGL f-\ii\bar c\bar Zf,Zf\rangle&=\ii \bar c -\ii\bar c =0.
\end{align*}
\end{proof}
\section{A classification result}
In the following we will keep the hypothesis that $n=4$. We give a complete classification for CR pluriharmonic isometric immersions of three-dimensional strongly pseudoconvex pseudohermitian CR manifolds in the Euclidean space $\mathbb{R}^4$.

\begin{Th}\label{th:classification}
Let $M$ be a three-dimensional strongly pseudo-convex pseudohermitian CR manifold, and $f\colon M\to\bR^4$ an isometric CR pluriharmonic immersion. Then one of the two following cases occurs:
\begin{itemize}
\item[1.] there is a CR diffeomorphism $\psi\colon M\to U$, where $U$ is an open subset of the cylinder $\tilde M=\{(z,w)\in\bC^2\mid(\Re z)^2+(\Re w)^2=1\}$, and an isometry $\phi\colon\bC^2\to\bR^4$, such that $f=\phi\circ\psi$;
\item[2.] there is a CR diffeomorphism $\psi\colon M\to U$, where $U$ is an open subset of the sphere $\tilde M=\{(z,w)\in\bC^2\mid |z|^2+|w|^2=2\}$, and an isometry $\phi\colon\bC^2\to\bR^4$, such that $f=\phi\circ\psi$.
\end{itemize}
\end{Th}

\begin{Rem}
In contrast with the result of \cite{APS}, stating that the induced metric on $M$ is always K\"ahler, in the present case the sphere is Sasakian, but the cylinder is not.
\end{Rem}

A global result easily follows.
\begin{Cor}
Let $M$ be a complete three-dimensional strongly pseudo-convex pseudohermitian CR manifold, and $f\colon M\to\bR^4$ an isometric CR pluriharmonic immersion. Then one of the two following cases occurs:
\begin{itemize}
\item[1.] if $M$ is not compact, there is a local CR diffeomorphism and covering map $\psi\colon M\to\tilde M$, where $\tilde M=\{(z,w)\in\bC^2\mid(\Re z)^2+(\Re w)^2=1\}$ is the cylinder, and an isometry $\phi\colon\bC^2\to\bR^4$, such that $f=\phi\circ\psi$;
\item[2.] if $M$ is compact, there is a CR diffeomorphism $\psi\colon M\to \tilde M$, where  $\tilde M=\{(z,w)\in\bC^2\mid |z|^2+|w|^2=2\}$ is the sphere, and an isometry $\phi\colon\bC^2\to\bR^4$, such that $f=\phi\circ\psi$.
\end{itemize}
\end{Cor}

The rest of this section is devoted to the proof of Theorem~\ref{th:classification}

Let us first summarize the properties of $\DGL f$ for a CR pluriharmonic isometric immersion $f\colon M\to\bR^4$:
\begin{enumerate}[(i)]
\item
$\langle \DGL f,\DGL f\rangle=0$,
\item
$\|\DGL f\|^2=\frac12$,
\item
$\langle X\DGL f, Y\DGL f\rangle = 0$ for all $X,Y\in TM$.
\end{enumerate}

We characterize the image of $\DGL f$.
Let $N=\{v\in\bC^4\mid \langle v,v\rangle=0,\ \|v\|^2=1/2\}$. The image of $\DGL f$ is then contained in $N$.
Let $v\in N$. Then
\[
T_{v}N=\{w\in\bC^4\mid\langle v,w\rangle=0,\ \Re(v,w)=0\}.
\]
The cone of isotropic vectors in $T_vN$ is
\[
T_v^0N=\{w\in T_vN\mid\langle w,w\rangle=0\}
\]
We decompose $T_v^0N$ as follows. Let $u\in\bC^4\setminus\{0\}$ be a vector such that
\begin{equation}\label{eq:a11}
\langle u,v\rangle=0,\quad\langle u,\bar v\rangle=0,\quad\langle u,u\rangle=0.
\end{equation}
Then $T_v^0N$ is the union of two three-dimensional real subspaces of $T_vN$:
\[
T_v^0N=(\bR\ii v+ \bC u) \cup (\bR\ii v+ \bC\bar u).
\]
This decomposition gives two smooth vector distributions $\mathcal D^\pm$ on $N$. Let $\mathcal D^+$ be the distribution containing the vector $e_3+\ii e_4\in T^0_{e_1+\ii e_2}N$ and  $\mathcal D^-$ be the distribution containing the vector $e_3-\ii e_4\in T^0_{e_1+\ii e_2}N$. Their intersection is the one-dimensional distribution  $(\mathcal D^+\cap\mathcal D^-)(v)=\bR\ii v$.

They are both integrable, and the integral manifolds are $3$-spheres. Fix a point $v\in N$ and a vector $u\in\mathcal D^+(v)$ with $\|u\|^2=1/2$ The integral manifolds through $v$ are
\begin{align*}
S^+&=\{\lambda v + \mu u\mid\lambda,\mu\in\bC,\ |\lambda|^2+|\mu|^2=1\},\\
S^-&=\{\lambda v + \mu \bar u\mid\lambda,\mu\in\bC,\ |\lambda|^2+|\mu|^2=1\}.
\end{align*}
Notice that both spheres $S\pm$ are contained in a complex $2$-plane.

\medskip

The tangent cone to the image of $\DGL f$ is contained in $T^0N$.

If the rank of $\DGL f$ is at least $2$, then the image of $\DGL f$ is locally contained in a sphere $S^\pm$. Since both $Z\DGL f$ and $\bar Z\DGL f$ are orthogonal to $\DGL f$, it follows that $Z\DGL f$ and $\bar Z\DGL f$  are linearly dependent over $\bC$.
If the rank of the differential of $\DGL f$ is  $0$ or $1$, trivially  $Z\DGL f$ and $\bar Z\DGL f$  are linearly dependent over $\bC$.
From:
\[
(Z\DGL f,\bar Z\DGL f)=\langle Z\DGL f,Z\DGLb f\rangle=\bar c\langle\bar Z f,ZTf\rangle=\frac{\ii\bar c}{2}
\]
we obtain that either $c=0$, i.e. $M$ is Sasakian, or $\bar Z\DGL f=\frac{\bar Z f}{2}$.
We consider those cases separately, additionally distinguishing the cases $|c|=\frac12$ and $|c|\neq\frac12$, completing the proof of the Theorem.

\paragraph{The case $\bar Z\DGL f=\frac{\bar Z f}{2}$, $|c|\neq\frac12$:\\}
\noindent We will show that this case reduces in fact to the Sasaki case (i.e $c=0$).
Assume for the moment that $|c|^2\neq 1/4$. Let $W=(1-4|c|^2)^{-1}(Z-2\ii\bar c\bar Z)$. Then
\begin{align*}
&&W\DGL f=0,\quad&&\bar W\DGL f=\frac{\bar Z f}{2},\quad&&(\bar W\DGL f, T\DGL f)=0.
\end{align*}
Let $$S=(1-4|c|^2)^{-1}T,\quad \eta=(1-4|c|^2)\theta,$$
and complete $\{\eta\}$ to a dual basis $\{\xi,\bar\xi,\eta\}$ of $\{W,\bar W,S\}$. Moreover, let $\hat d,\hat D, \hat d',\hat d'', \mathrm{etc.}$ be the differential operators associated to this CR structure.\smallskip

Denote by $M'$ the manifold $M$ with the pseudohermitian structure given by $\eta$ and $T^{1,0}M'=\bC W$. Then $\DGLb f$ is a CR map and (up to a pseudo-homothety) a local isometric diffeomorphism from $M'$ to the sphere $S^\pm$, and hence $M'$ is a Sasaki manifold with respect to $\eta$, and $S$ is the Reeb vector field on $M'$.

The CR manifolds $M$ and $M'$ have the same contact distribution and proportional Reeb vector fields. It follows that $S$ is a constant multiple of $T$, i.e.\ that $|c|$ is constant.
Up to a conformal transformation $Z'=\ee^{\ii v}Z$ (see section 1), we can consequently assume that $c$ is actually constant.

The pseudohermitian structure of $M'$ is Sasakian, and then $[W,S]$ is a multiple of $W$. We have:
\begin{align*}
[W,S]&=(1-4|c|^2)^{-2}(b-4\ii |c|^2+4cT\bar c+4b|c|^2)W\\
         &\quad+(1-4|c|^2)^{-2}(4\ii b\bar c+4\bar c|c|^2+\bar c+2\ii T\bar c)\bar W.
\end{align*}
Taking in consideration that $c$ is constant and nonzero we obtain
\[
b=\ii(|c|^2+1/4)
\]
and
\[
[W,S]=\frac\ii4W.
\]
We also observe that, under the hypothesis that $c$ is constant, we have:
\begin{align*}
\ii[W,\bar W]&=(1-4|c|^2)^{-1}(T+aZ+\bar a\bar Z)\\
  &=S+\frac{a-2\ii\bar ac}{1-4|c|^2}W+\frac{\bar a+2\ii a\bar c}{1-4|c|^2}\bar W.
\end{align*}
The structure functions of the manifold $M'$ are hence
\[
\alpha=\frac{a-2\ii\bar ac}{1-4|c|^2},\qquad\beta=\frac\ii4,\qquad\gamma=0.
\]

We want to prove that $\alpha=0$.\\
 We first observe that, since $\bar W\DGL f=\frac{\bar Z f}{2}$, and $d''f=\frac{\bar Z}{2}\bar\zeta$, we have $$D(\bar W\DGL f\bar\zeta)=Dd''f=0,$$
  because of the CR pluriharmonicity of $f$.\\
  On the other hand, since $W\DGL f=0$, we have $0=\hat D\hat d(\DGL f)=\hat D(\bar W\DGL f\bar\xi)$. We note that $D$ and $\hat D$ are constant multiples of each other, because $\eta$ is a constant multiple of $\theta$. As a consequence, since $D$ depends only on the contact structure $\theta$ (see section 2), we obtain that
  $$\hat D(\bar W\DGL f\bar\zeta)=0,\quad\hat D(\bar W\DGL f\zeta)=0,\textrm{ and finally }\hat D(\bar W\DGL f\xi)=0.$$ In particular:
\[\hat{D}'(\bar W\DGL f\xi)=
S \bar W\DGL f+\beta \bar W\DGL f+\ii W(\bar W\bar W\DGL f-\ii\alpha \bar W\DGL f)=0.
\]
Using the CR pluriharmonicity of $f$, and again the fact that $\bar W\DGL f=\frac{\bar Z f}{2}$ and $W\DGL f=0$, we first compute $\bar W\bar W\DGL f-\ii\alpha W\DGL f$:
\begin{align*}
\langle\bar W\bar W\DGL f-\ii\alpha\bar  W\DGL f,\DGL f\rangle&=0\\
\langle\bar W\bar W\DGL f-\ii\alpha\bar  W\DGL f,\DGLb f\rangle&=0\\
\langle\bar W\bar W\DGL f-\ii\alpha\bar  W\DGL f,\bar W\DGL f\rangle&=0\\
\langle\bar W\bar W\DGL f-\ii\alpha\bar  W\DGL f,W\DGLb f\rangle&=-2\ii\alpha\langle\bar  W\DGL f,W\DGLb f\rangle
\end{align*}
that is $\bar W\bar W\DGL f-\ii\alpha \bar W\DGL f=-2\ii\alpha\bar  W\DGL f$. Then we compute
\begin{align*}
&\langle S \bar W\DGL f+\beta \bar W\DGL f+\ii W(\bar W\bar W\DGL f-\ii\alpha \bar W\DGL f),\DGLb f\rangle=\\
&\quad =\langle S \bar W\DGL f+\beta \bar W\DGL f+2 W(\alpha \bar W\DGL f),\DGLb f\rangle\\
&\quad =\langle 2 W\alpha\bar W\DGL f+2 \alpha W\bar W\DGL f,\DGLb f\rangle\\
&\quad =\langle 2 \alpha W\bar W\DGL f,\DGLb f\rangle=-2\alpha\langle\bar W\DGL f,W\DGLb f\rangle
\end{align*}
thus proving that $\alpha=0$. Since $|c|\neq1/2$, this shows that $a=0$.\\
\smallskip

 Let $X=2^{-1/2}(Z+\bar Z)$, $Y=2^{-1/2}\ii(Z-\bar Z)$. We can further assume up to a conformal change (see section 2.1) that $c$ is purely imaginary. Then we have:
\begin{align}
\begin{split}X\Delta_R f&=
X\DGL f+X\DGLb f\\ &=2^{-3/2}(1-2\ii c)(Zf+\bar Z f)=\frac{1-2\ii c}{2}Xf,
\end{split}\\
\begin{split}
Y\Delta_R f&=
Y\DGL f+Y\DGLb f\\ &=2^{-3/2}\ii(1+2\ii c)(\bar Zf-Z f)=\frac{1+2\ii c}{2}Yf.
\end{split}\end{align}
Furthermore, using the brackets \eqref{eq_for_a}, 
we get
\begin{align*}
\langle T\DGL f, Zf\rangle&=0,\\
\langle T\DGL f, \bar Zf\rangle&=0,\\
\langle T\DGL f, \DGL f\rangle&=0,\\
\langle T\DGL f, \DGLb f\rangle&=\ii\langle Z\bar Z\DGL f,\DGLb f\rangle-\ii\langle\bar ZZ\DGL f,\DGLb f\rangle\\ &=-\ii\|\bar Z\DGL f\|^2+\ii\|Z\DGL f\|^2=-\frac{\ii}{4}(1-4|c|^2),
\end{align*}
which implies, using the fact that $\|\DGL f\|^2=\frac{1}{2}$ and equation \eqref{relation_between_D_R_and_D_GL}
\begin{equation}
\begin{aligned}
T\DGL f&=-\frac{\ii}{2}(1-4|c|^2)\DGL f,\\
T\Delta_R f&=\frac{1-4|c|^2}{2} Tf.
\end{aligned}
\end{equation}
Then as $\DR$ is the unit normal vector, the directions defined by $X$, $Y$ and $T$ are principal curvatures directions for $f(M)$ and the principal curvatures are constant and equal to:
\[ \frac{1-2\ii c}{2},\qquad \frac{1+2\ii c}{2}, \qquad \frac{1+4c^2}{2}. \]

A result of B.~Segre (\cite{S}, see also \cite{C}) implies that either all the three curvatures are equal, or at least one is zero. The latter case contradicts the condition $|c|\neq1/2$. The former case corresponds to the condition $c=0$, which is analyzed below.

\paragraph{The case $\bar Z\DGL f=\frac{\bar Z f}{2}$, $|c|=1/2$\\}
We can assume up to conformal changes (see section 2.1) that $c$ is constant and equal to $\ii/2$. Let $X=2^{-1/2}(Z+\bar Z)$, $Y=2^{-1/2}\ii(Z-\bar Z)$. Then:
\[
Y\DGL f=0,\qquad X\DGL f=2^{-1/2}\bar Z f.
\]
Separating real and imaginary part of $Y\DGL f=0$, we obtain that $Y\Delta_R f=0$ and $YTf=0$.
We compute now that $T\DGL f=0$. We have using \eqref{eq_for_a} and the CR pluriharmonicity:
\begin{align*}
\langle T\DGL f,Zf\rangle&=0,\\
\langle T\DGL f,\bar Zf\rangle&=0,\\
\langle T\DGL f,\DGL f\rangle&=0,\\
\langle T\DGL f,\DGLb f\rangle&=\langle \ii Z\bar Z\DGL f -\ii \bar Z Z\DGL f,\DGLb f\rangle\\
  &=\frac\ii2\langle Z\bar Z f,\DGLb f\rangle+\bar c\langle\bar Z\bar Z f,\DGLb f\rangle\\
  &=-\frac\ii4+\frac\ii4=0.
\end{align*}
It follows that $T\Delta_R f=0$ and $TTf=0$.
On the other hand we have $[Y,T]=-\ii(b-c)X$. Since $X\DGL f=\frac{Xf}2-\ii\frac{Yf}2\neq 0$ and $[Y,T]\DGL f=0$, we obtain
\begin{equation}
[Y,T]=0,\qquad b=c=\frac\ii2, \qquad TYf=0.
\end{equation}
The distribution generated by $Y$ and $T$ is integrable, and hence determines a foliation $\mathcal F$ of $M$. $Y$, $T$ are orthogonal and commute and this gives a system of local coordinates, so that the leaves of $\mathcal F$ are locally isometric to $\bR ^2$. On each leave $F$ of $\mathcal F$ the vector fields $T$ and $Y$ are parallel, and $Tf$, $\Delta_R f$ are constant. It follows that the image $f(F)$ of every leave $F$ is contained in the three-dimensional affine plane orthogonal to $\Delta_R f$ at any point of $f(F)$.\smallskip

First we observe that $[X,Y]=-T-2^{-1/2}(a+\bar a)X-2^{-1/2}\ii(a-\bar a)Y$.
As $\langle YYf,Yf\rangle=\langle YYf,\DGL f\rangle=\langle YYf,Tf\rangle=0$, the norm of $YYf$ is given by
\begin{equation}
\|YYf\|^2=\langle YYf,Xf\rangle^2=\langle Yf,[X,Y]f\rangle^2=-\frac{(a-\bar a)^2}2
\end{equation}
and $T\|YYf\|=0$, yielding $T(a-\bar a)=0$. As $b$ and $c$ are constant equal to $\frac{i}{2}$, we obtain from \eqref{eq_jacobi} the condition $Ta-ab-\bar ac=0$ and finally
$T(a+\bar a)=0$ and $T(a-\bar a)=\ii(a+\bar a)$. Comparing these conditions we have
\[
a+\bar a=0,\qquad Ta=0,\qquad [X,Y]=-T-2^{1/2}\ii aY.
\]

We compute now the derivatives of f of the form $X^kf$. We use that $[X,T]=Y$ and $X\Delta_Rf=Xf$. We have
\begin{align*}
(XXf,Xf)&=0&(XXf,Yf)&=-(Xf,[X,Y]f)=0\\
(XXf,Tf)&=0& (XXf,\Delta_R f)&=-(Xf,X\Delta_R f)=-1
\end{align*}
and then
\begin{equation}
XXf=-\Delta_R f,\quad XXXf=-Xf,\quad X^4f=\Delta_R f.
\end{equation}
Hence the integral lines of $Xf$ have vanishing higher order curvatures, and are (arcs of) circles of radius $1$ in the affine plane spanned by $Xf$ and $\Delta_R f$.
The planes generated by $Yf$, $Tf$ are thus constant along $X$, and also along $T$, and consequently along $Y=[X,T]$. In other words the leaves of $\mathcal F$ are affine planes.

We can conclude now that $f(M)$ is, up to rigid motions of $\bR^4$, equal to (an open subset of) the tube $\{x_1^2+x_4^2=1\}$. The CR structure on $M$ is given by the restriction to $f(M)$ of the complex structure of $\bR^4$ defined by $J e_1=e_2$, $Je_3=e_4$ (or the conjugate one).

\paragraph{The Sasakian case ($c=0$):\\}
In this case we have $Z\DGL f=0$ and $\bar Z\DGL f=\ii\bar Z T f$. We obtain, using $\|\DGL f\|^2=\frac{1}{2}$:
\begin{align*}
\|\bar Z\DGL f\|^2&=\langle\bar Z\DGL f,Zf\rangle\langle\bar Zf, Z\DGLb f\rangle=\|\DGL f\|^4=\frac14,\\
(\bar Z\DGL f,T\DGL f)&=\langle\bar Z\DGL f, Zf\rangle\langle\bar Zf ,T\DGLb f\rangle+\langle\bar Z\DGL f, \bar Zf\rangle\langle Zf,T\DGLb f\rangle\\
&\quad+2\langle\bar Z\DGL f,\DGL f \rangle\langle\DGLb f ,T\DGLb f\rangle\\
&\quad+2\langle\bar Z\DGL f, \DGLb f\rangle\langle\DGL f ,T\DGLb f\rangle=0,\\
\|T\DGL f\|^2&=2\langle T\DGL f,\DGLb f\rangle\langle T\DGLb f,\DGL f\rangle\\
&=2\langle\ii Z\bar Z\DGL f,\DGLb f\rangle\langle\DGL f,-\ii\bar Z Z\DGLb f\rangle\\
&=2\|\bar Z\DGL f\|^4=\frac18.
\end{align*}
It follows that $2^{3/2}\DGLb f$ is a CR map and an isometric immersion of $M$ in a standard sphere of radius $2$ in $\bC^2$.

Moreover
\begin{align*}
\langle\bar Z\DGL f,Zf\rangle&=\|\DGL f\|^2=1/2,&\langle\bar Z\DGL f,\bar Zf\rangle&=0\\
\langle\bar Z\DGL f,\DGL f\rangle&=0,&\langle\bar Z\DGL f,\DGLb f\rangle&=0.
\end{align*}
 hence $\bar Z\DGL f=\frac{\bar Zf}2$. Together with $\bar Z\DGLb f=0$ this gives $\bar Z(\Delta_R f-\frac f2)=0$. Since $\Delta_R f-\frac f2$ is real, it must be constant, i.e.
\begin{equation}
f=2\Delta_R f + \mathrm{const}
\end{equation}
and $f$ is the standard embedding of a sphere of radius $2$.

\end{document}